\documentclass[11pt]{amsart}
\usepackage[margin=1in]{geometry}
\usepackage[T1]{fontenc}
\usepackage[utf8]{inputenc}
\usepackage{lmodern}
\usepackage{microtype}
\usepackage{amsmath,amssymb,amsthm,mathtools,latexsym}
\usepackage{tikz}
\usetikzlibrary{arrows.meta,positioning,calc}
\usepackage[colorlinks=true,linkcolor=blue!60!black,urlcolor=blue!60!black,citecolor=blue!60!black]{hyperref}

\newtheorem{theorem}{Theorem}
\newtheorem{lemma}{Lemma}

\theoremstyle{definition}

\begin{document}

\title[An elementary proof of the Chebyshev bounds]{An Elementary Proof of the Chebyshev Bounds via a Weighted Prime Sum}

\author{Kai Hubbard}
\address{Department of Mathematics, Pasadena City College, Pasadena, CA 91106, USA}
\email{kaihubbard520@gmail.com}

\date{\today}

\subjclass[2020]{11N05 (primary), 11A41 (secondary)}
\keywords{Chebyshev bounds, prime-counting function, elementary proof, Mertens' theorem}

\begin{abstract}
The Chebyshev bounds for the prime-counting function, i.e., $\pi(x) \asymp x/\ln x$, is established in a new way. This new approach follows the outline
\[
\pi(x) \;\asymp\; \Bigl(\sum_{p \le x} \sqrt{\tfrac{\ln p}{p}}\Bigr)^{\!2} \;\asymp\; \frac{x}{\ln x}.
\]
Here, the second $\asymp$ is derived from the classical estimate by Mertens, i.e., $\sum_{p \le x} \frac{\ln p}{p} = \ln x + O(1)$; while the first $\asymp$ is proved by considering the difference
\[
\Bigl(\sum_{p \le x} \sqrt{\tfrac{\ln p}{p}}\Bigr)^{\!2} \;-\; \sum_{p \le x} \frac{\ln p}{p},
\]
which is shown as having the same order as $\pi(x)$.
\end{abstract}

\maketitle

\section{Introduction}

The Chebyshev bounds asserts that $\pi(x) \asymp x/\ln x$, where $\pi(x)$ denotes the number of primes less than or equal to $x$.  These bounds were first established by Chebyshev~\cite{Chebyshev} in 1852, and they represent the strongest statement about the distribution of primes that can be obtained without complex analysis.  The bounds were later sharpened to the asymptotic $\pi(x) \sim x/\ln x$ (the prime number theorem) by Hadamard~\cite{Hadamard} and de la Vall\'ee Poussin~\cite{Vallee} in 1896.  An elementary proof of the prime number theorem was given by Selberg~\cite{Selberg} and Erd\H{o}s~\cite{Erdos} in 1949.


Instead of considering Mertens' classical summation $\sum_{p \le x} \frac{\ln p}{p}$, we study the square of a closely related sum, i.e.\ $\bigl(\sum_{p \le x} \sqrt{\ln p / p}\bigr)^{2}$, and finally establish the Chebyshev's bounds for $\pi(x)$ from our study.  The logical deduction of the argument is summarized in Figure~\ref{fig:logictree}.

\begin{figure}[h]
\centering
\begin{tikzpicture}[font=\footnotesize,>={Latex[length=2.5mm]}]
  \node (top) at (0,3.6) {$\pi(x) \asymp \dfrac{x}{\ln x}$};
  \node (left) at (-3.8,1.8) {$\Bigl(\sum_{p \le x} \sqrt{\tfrac{\ln p}{p}}\Bigr)^{\!2} \asymp \dfrac{x}{\ln x}$};
  \node (right) at (3.6,1.8) {$\pi(x) \asymp \Bigl(\sum_{p \le x} \sqrt{\tfrac{\ln p}{p}}\Bigr)^{\!2}$};
  \node (mid) at (3.6,0.0) {$\pi(x) \asymp \Bigl(\sum_{p \le x} \sqrt{\tfrac{\ln p}{p}}\Bigr)^{\!2} - \sum_{p \le x} \dfrac{\ln p}{p}$};
  \node (bot) at (-3.8,-1.5) {$\sum_{p \le x} \dfrac{\ln p}{p} = \ln x + O(1)$ \;(Mertens' estimate)};
  \draw[->] (left)  -- (top);
  \draw[->] (right) -- (top);
  \draw[->] (mid)   -- (right);
  \draw[->] (left)  -- (mid);
  \draw[->] (bot)   -- (left);
\end{tikzpicture}
\caption{Deduction graph of the argument.}
\label{fig:logictree}
\end{figure}

For the ease of expression, we need some convention of notation throughout this paper.

\section{Notation}

Throughout this paper, the variable $x$ is supposed to be greater than or equal to $2$.  This is to make $1/\ln x$ meaningful.

The relation $f(x) \asymp g(x)$ means $f$ and $g$ have the same order of magnitude, i.e., there exist positive constants $k$ and $K$ (independent of $x$) such that $k \cdot g(x) \le f(x) \le K \cdot g(x)$.  Clearly, if $f(x) \asymp g(x)$, then $g(x) \asymp f(x)$, i.e., the relation $\asymp$ is symmetric.

Throughout this paper, we use the following notations:
\begin{itemize}
\item $\pi(x) = $ the number of primes less than or equal to $x$;
\item $p$ represents a prime (according to context);
\item $p_n = $ the $n$-th prime;
\item $M(x) = \sum_{p \le x} \dfrac{\ln p}{p}$ ($=$ Mertens' summation);
\item $S(x) = \sum_{p \le x} \sqrt{\dfrac{\ln p}{p}}$ ($=$ the summation we study);
\item $E(x) = \bigl(\sum_{p \le x} \sqrt{\ln p / p}\bigr)^{2} - \sum_{p \le x} \frac{\ln p}{p} = S^{2}(x) - M(x)$ ($=$ the difference we study);
\item $h(x) = \sqrt{\dfrac{x}{\ln x}}$ ($=$ the square root of the asymptotic function $x/\ln x$).
\end{itemize}

Notice that
\[
h'(x) \;=\; \frac{1}{2} \cdot \frac{1}{\sqrt{x/\ln x}} \cdot \frac{1 \cdot \ln x - x \cdot \frac{1}{x}}{(\ln x)^{2}} \;>\; 0
\qquad (\text{for } x \ge 3),
\]
i.e., $h(x) = \sqrt{x/\ln x}$ is increasing, and so $1/h(x) = \sqrt{\ln x / x}$ is decreasing.

Under this notation, the outline of our study becomes
\begin{equation}
\left.
\begin{aligned}
& S^{2}(x) \;\asymp\; \dfrac{x}{\ln x} \\
& \Downarrow \\
& \pi(x) \;\asymp\; S^{2}(x) - M(x) \\
& \Downarrow \\
& \pi(x) \;\asymp\; S^{2}(x)
\end{aligned}
\right\}
\;\;\Longrightarrow\;\;
\pi(x) \;\asymp\; \dfrac{x}{\ln x}.
\label{eq:outline}
\end{equation}
Now we establish the above diagram step by step.  We begin with the first part of our outline in~(\ref{eq:outline}).

\section{The Estimate $S^{2}(x) \asymp x/\ln x$}

This conclusion is equivalent to $S(x) \asymp \sqrt{x/\ln x}$, which is, in turn, equivalent to the following Lemma~\ref{lem:upper} and Lemma~\ref{lem:lower}.

\begin{lemma}\label{lem:upper}
$S(x) = O\bigl(\sqrt{x/\ln x}\bigr)$.
\end{lemma}

\begin{lemma}\label{lem:lower}
$S(x) > k \cdot \sqrt{x/\ln x}$ for some constant $k$.
\end{lemma}

To prove Lemma~\ref{lem:upper}, we need the following two famous results.

\begin{lemma}[Abel's identity, see Theorem~4.2 in~\cite{Apostol}]\label{lem:abel}
For any arithmetical function $a(n)$ with $a(1)=0$ and any continuously differentiable real function $f(x)$, we have
\[
\sum_{n \le x} a(n) f(n) \;=\; \Bigl(\sum_{n \le x} a(n)\Bigr) f(x) \;-\; \int_{2}^{x} \Bigl(\sum_{n \le t} a(n)\Bigr) f'(t)\, dt.
\]
\end{lemma}

\noindent
Clearly, this conclusion is also true if $\sum_{n \le x}$ is replaced by $\sum_{p \le x}$, since we could define $a(n)=0$ for all composite number $n$'s.

\begin{lemma}[Mertens' estimate, see~\cite{Mertens} or Theorem~4.10 in~\cite{Apostol}]\label{lem:mertens}
\[
\sum_{p \le x} \frac{\ln p}{p} \;=\; \ln x + O(1).
\]
\end{lemma}

Now we are ready to prove Lemma~\ref{lem:upper}.

\begin{proof}[Proof of Lemma~\ref{lem:upper}]
We have
\begin{align*}
S(x) \;&=\; \sum_{p \le x} \sqrt{\frac{\ln p}{p}} \\
     \;&=\; \sum_{p \le x} \frac{\ln p}{p} \cdot \sqrt{\frac{p}{\ln p}} \\
     \;&=\; \Bigl(\sum_{p \le x} \frac{\ln p}{p}\Bigr) h(x) - \int_{2}^{x} \Bigl(\sum_{p \le t} \frac{\ln p}{p}\Bigr) h'(t)\, dt && (\text{by Lemma~\ref{lem:abel}})\\
     \;&=\; M(x) h(x) - \int_{2}^{x} M(t) h'(t)\, dt \\
     \;&=\; \bigl(\ln x + O(1)\bigr) h(x) - \int_{2}^{x} \bigl(\ln t + O(1)\bigr) h'(t)\, dt && (\text{by Lemma~\ref{lem:mertens}})\\
     \;&=\; h(x) \ln x + O\bigl(h(x)\bigr) - \int_{2}^{x} \ln t \cdot h'(t)\, dt + O\!\left(\int_{2}^{x} h'(t)\, dt\right) \\
     \;&=\; \sqrt{\frac{x}{\ln x}} \cdot \ln x - \int_{2}^{x} \ln t \cdot h'(t)\, dt + O\bigl(h(x)\bigr) + O\bigl(h(x)-h(2)\bigr) \\
     \;&=\; \sqrt{x \cdot \ln x} - \int_{2}^{x} \ln t \cdot h'(t)\, dt + O\bigl(h(x)\bigr).
\end{align*}
We rewrite the above as
\begin{equation}
S(x) \;=\; \sqrt{x \cdot \ln x} \;-\; \int_{2}^{x} \ln t \cdot h'(t)\, dt \;+\; O\bigl(h(x)\bigr). \label{eq:Sx}
\end{equation}

Now we estimate the integral in the middle:
\begin{align*}
\int_{2}^{x} \ln t \cdot h'(t)\, dt
\;&=\; -\int_{2}^{x} \frac{h(t)}{t}\, dt + \ln t \cdot h(t)\Big|_{2}^{x} && (\text{integration by parts}) \\
\;&=\; -\int_{2}^{x} \frac{\sqrt{t / \ln t}}{t}\, dt + \ln t \cdot \sqrt{\frac{t}{\ln t}}\,\Big|_{2}^{x} \\
\;&=\; -\int_{2}^{x} \frac{1}{\sqrt{t \cdot \ln t}}\, dt + \sqrt{t \cdot \ln t}\,\Big|_{2}^{x} \\
\;&=\; -\int_{2}^{x} \frac{1}{\sqrt{t \cdot \ln t}}\, dt + \sqrt{x \cdot \ln x} \;+\; \text{constant} \\
\;&=\; O\!\left(\sqrt{\frac{x}{\ln x}}\right) + \sqrt{x \cdot \ln x} + \text{constant}.
\end{align*}
That is,
\begin{equation}
\int_{2}^{x} \ln t \cdot h'(t)\, dt \;=\; O\!\left(\sqrt{x/\ln x}\right) + \sqrt{x \cdot \ln x} + \text{constant}. \label{eq:integral}
\end{equation}
The last equality is derived from the following limit, computed by L'Hospital's rule:
\[
\lim_{x \to \infty} \frac{\int_{2}^{x} \frac{1}{\sqrt{t \cdot \ln t}}\, dt}{\sqrt{x / \ln x}}
\;=\; \lim_{x \to \infty} \frac{\dfrac{1}{\sqrt{x \cdot \ln x}}}{\dfrac{1}{2} \cdot \dfrac{1}{\sqrt{x/\ln x}} \cdot \dfrac{\ln x - 1}{(\ln x)^{2}}}
\;=\; \lim_{x \to \infty} \frac{2}{1 - \frac{1}{\ln x}}
\;=\; 2.
\]

Lastly, plugging Equation~(\ref{eq:integral}) into Equation~(\ref{eq:Sx}), we get
\begin{align*}
S(x) \;&=\; \sqrt{x \cdot \ln x} - \int_{2}^{x} \ln t \cdot h'(t)\, dt + O\bigl(h(x)\bigr) \\
     \;&=\; \sqrt{x \cdot \ln x} - \Bigl(O\bigl(\sqrt{x/\ln x}\bigr) + \sqrt{x \cdot \ln x} + \text{constant}\Bigr) + O\!\left(\sqrt{\frac{x}{\ln x}}\right) \\
     \;&=\; O\!\left(\sqrt{x/\ln x}\right).
\end{align*}
This completes the proof of Lemma~\ref{lem:upper}.
\end{proof}

Now let us focus on Lemma~\ref{lem:lower}.

\begin{proof}[Proof of Lemma~\ref{lem:lower}]
For any fixed constant $c$, consider the partial sum
\[
S(x) - S(x/c) \;=\; \sum_{x/c < p \le x} \sqrt{\frac{\ln p}{p}}.
\]
Since $x/c < p$, we have
\[
\sqrt{\frac{\ln p}{p}} \;\le\; \sqrt{\frac{\ln(x/c)}{x/c}}
\qquad \bigl(\sqrt{\ln x / x} = 1/h(x) \text{ is decreasing}\bigr),
\]
so
\[
\frac{\ln p}{p} \;\le\; \sqrt{\frac{\ln(x/c)}{x/c}} \cdot \sqrt{\frac{\ln p}{p}},
\]
thus
\[
\sum_{x/c < p \le x} \frac{\ln p}{p}
\;\le\; \sqrt{\frac{\ln(x/c)}{x/c}} \cdot \Bigl(\sum_{x/c < p \le x} \sqrt{\frac{\ln p}{p}}\Bigr),
\]
i.e.,
\[
\frac{\sum_{x/c < p \le x} \frac{\ln p}{p}}{\sqrt{\ln(x/c)/(x/c)}}
\;\le\; \sum_{x/c < p \le x} \sqrt{\frac{\ln p}{p}}
\;=\; S(x) - S(x/c) \;<\; S(x).
\]
We can choose a suitable $c$ such that $\sum_{x/c < p \le x} \frac{\ln p}{p} > 1$ (see~(\ref{eq:gap}) at the end of this proof).  Then we have
\[
S(x) \;>\; \frac{\sum_{x/c < p \le x} \frac{\ln p}{p}}{\sqrt{\ln(x/c)/(x/c)}}
\;>\; \frac{1}{\sqrt{\ln(x/c)/(x/c)}}
\;=\; \sqrt{\frac{x/c}{\ln(x/c)}}
\;=\; \frac{1}{\sqrt{c}} \cdot \sqrt{\frac{x}{\ln x - \ln c}}
\;>\; \frac{1}{\sqrt{c}} \cdot \sqrt{\frac{x}{\ln x}}.
\]
This is the conclusion of Lemma~\ref{lem:lower}.

Finally, we show that a constant $c$ could be chosen such that $\sum_{x/c < p \le x} \frac{\ln p}{p} > 1$.  In fact,
\begin{equation}
\sum_{x/c < p \le x} \frac{\ln p}{p} \;=\; M(x) - M(x/c) \;=\; \ln x - \ln(x/c) + O(1) \;=\; \ln c + O(1), \label{eq:gap}
\end{equation}
where the second equality is by Lemma~\ref{lem:mertens}.  Clearly, we can choose $c$ such that
\[
\sum_{x/c < p \le x} \frac{\ln p}{p} \;=\; \ln c + O(1) \;>\; 1
\]
(since this $O(1)$ does \emph{not} depend on $c$).  The proof of Lemma~\ref{lem:lower} is completed.
\end{proof}

Now that the first part of the outline in~(\ref{eq:outline}), i.e., $S^{2}(x) \asymp x/\ln x$, is established, we can proceed to its second part.

\section{The Estimate $\pi(x) \asymp S^{2}(x) - M(x)$}

\begin{lemma}\label{lem:pi}
$\pi(x) \asymp S^{2}(x) - M(x)$.
\end{lemma}

\begin{proof}[Proof of Lemma~\ref{lem:pi}]
The main idea is to study the difference function
\[
E(x) \;=\; S^{2}(x) - M(x).
\]
Recall that $p_n$ denotes the $n$-th prime.  And, for convenience, we denote the evaluations $S(p_n)$, $M(p_n)$, and $E(p_n)$ by $S_n$, $M_n$, and $E_n$, respectively.

Then
\[
\sum_{p \le p_n} \sqrt{\frac{\ln p}{p}}
\;=\; \sum_{p \le p_{n-1}} \sqrt{\frac{\ln p}{p}} + \sqrt{\frac{\ln p_n}{p_n}},
\]
i.e., $S(p_n) = S(p_{n-1}) + a_n$, where $a_n = \sqrt{\ln p_n / p_n}$, i.e., $S_n = S_{n-1} + a_n$.  Similarly, $M_n = M_{n-1} + a_n^{2}$.  Therefore
\begin{align*}
S_n^{2} - M_n \;&=\; (S_{n-1} + a_n)^{2} - (M_{n-1} + a_n^{2}) \\
               \;&=\; S_{n-1}^{2} + 2 a_n S_{n-1} + a_n^{2} - M_{n-1} - a_n^{2} \\
               \;&=\; S_{n-1}^{2} - M_{n-1} + 2 a_n S_{n-1},
\end{align*}
i.e., $E_n = E_{n-1} + 2 a_n S_{n-1}$.  Hence
\begin{align*}
E_n     \;&=\; E_{n-1} + 2 a_n     S_{n-1}, \\
E_{n-1} \;&=\; E_{n-2} + 2 a_{n-1} S_{n-2}, \\
        \;&\;\;\vdots \\
E_2     \;&=\; E_1 + 2 a_2 S_1.
\end{align*}
Summing up all the above equalities, we get
\begin{equation}
E_n \;=\; E_1 + 2 a_n S_{n-1} + 2 a_{n-1} S_{n-2} + \cdots + 2 a_2 S_1. \label{eq:En}
\end{equation}

Now we estimate $2 a_n S_{n-1}$.  Recall that, by Lemma~\ref{lem:upper}, we have $S(x) \asymp \sqrt{x/\ln x}$.  So we have
\[
\frac{a_n}{S_n} \;\asymp\; \frac{\sqrt{\ln p_n / p_n}}{\sqrt{p_n / \ln p_n}}
\;=\; \frac{\ln p_n}{p_n} \;\longrightarrow\; 0
\qquad (\text{as } n \to \infty),
\]
thus $a_n / S_n \ll 1$, so $a_n \ll S_n$, and, consequently, $S_{n-1} = S_n - a_n \asymp S_n$.

On the other hand, since $S(x) \asymp \sqrt{x/\ln x}$, we have $S_n \asymp \sqrt{p_n / \ln p_n} = 1/a_n$, thus $a_n S_n \asymp 1$.  Combining the above observation, i.e., $S_{n-1} \asymp S_n$, we conclude that
\[
a_n S_{n-1} \;\asymp\; a_n S_n \;\asymp\; 1,
\]
i.e., there exist positive constants $k$ and $K$ such that $k \le a_n S_{n-1} \le K$ (holds for any $n$ sufficiently large).

Putting this result into Equation~(\ref{eq:En}), we conclude that $E_n \asymp n$, i.e.,
\[
S_n^{2} - M_n \;\asymp\; n,
\]
i.e., $S^{2}(p_n) - M(p_n) \asymp n = \pi(p_n)$.  Hence
\[
S^{2}(x) - M(x) \;\asymp\; \pi(x).
\]
The proof of Lemma~\ref{lem:pi} is completed.
\end{proof}

\section{Summary and the Chebyshev Bounds}

In the previous section, we see that
\[
\pi(x) \;\asymp\; S^{2}(x) - M(x).
\]
This implies $\pi(x) \asymp S^{2}(x)$.  (Notice that $M(x) = \ln x + O(1) = o(x/\ln x) = o(S^{2}(x))$, where the last equality is by $S^{2}(x) \asymp x/\ln x$.)

Combining $\pi(x) \asymp S^{2}(x)$ and $S^{2}(x) \asymp x/\ln x$ (Lemma~\ref{lem:upper} together with Lemma~\ref{lem:lower}), we finally have the following theorem.

\begin{theorem}[Chebyshev bounds]
\[
\pi(x) \;\asymp\; \dfrac{x}{\ln x}.
\]
\end{theorem}

\section*{Acknowledgements}

The author thanks Dr.\ Buliao Wang for his sustained mentorship in analytic number theory and for checking an early version of this proof.

\end{document}